\documentclass[reqno, 11pt]{amsart}
\usepackage{amsmath,amsfonts,amssymb,amsthm,epsfig,epstopdf,url,array,comment,hyperref,mathrsfs,mathtools}
\usepackage{hyperref}
\usepackage{setspace}
\usepackage{graphicx}
\usepackage{float}
\usepackage{tikz-cd}
\usepackage[cmtip,all]{xy}
\newcommand{\longsquiggly}{\xymatrix{{}\ar@{<~>}[r]&{}}}
\usepackage{tabularray}
\mathtoolsset{showonlyrefs}

\theoremstyle{plain}
\newtheorem{thm}{Theorem}[section]

\newtheorem{prop}[thm]{Proposition}
\newtheorem{cor}[thm]{Corollary}

\theoremstyle{definition}
\newtheorem{definition}[thm]{Definition}
\newtheorem{con}[thm]{Conjecture}
\newtheorem{exmpl}[thm]{Example}

\theoremstyle{remark}
\newtheorem{rem}{Remark}

\newcommand{\ind}{\mathbf 1}

\newcommand{\R}{\mathbb{R}}
\renewcommand{\d}{\textnormal{ d}}

\newcommand{\lip}{\textnormal{Lip}}

\newcommand{\grad}{\nabla}

\newcommand{\ball}{\mathcal{B}}

\newcommand{\hidethis}[1]{}

\newcommand{\E}{\mathbb{E}}
\newcommand{\support}{\textnormal{supp}}

\newcommand{\conv}{\textnormal{Conv}}
\newcommand{\supp}{\textnormal{support}}

\newcommand{\vol}{\textnormal{vol}}

\newcommand{\dom}{\textnormal{Dom}}

\begin{document}

\title{The Kneser--Poulsen phenomena for entropy}
\author{Gautam Aishwarya}
\address{Technion Israel Institute of Technology, Faculty of Mathematics, Technion City, Haifa 3200003, Israel.}
\email{gautama@campus.technion.ac.il}
\author{Dongbin Li}
\address{University of Alberta, Department of Mathematical and Statistical Sciences, Edmonton, AB T6G 2R3, Canada.}
\email{dongbin@ualberta.ca}
\begin{abstract}
     The Kneser--Poulsen conjecture asserts that the volume of a union of balls in Euclidean space cannot be increased by bringing their centres pairwise closer. We prove that its natural information-theoretic counterpart is true. This follows from a complete answer to a question asked in \cite{AishwaryaAlamLiMyroshnychenkoZatarainVera23} about Gaussian convolutions, namely that the R\'enyi entropy comparisons between a probability measure and its contractive image are preserved when both undergo simultaneous heat flow. An inequality that unifies Costa's result on the concavity of entropy power with the entropic Kneser--Poulsen theorem is also presented.
    
\end{abstract}
\thanks{ 
	\\\indent GA was supported by ISF grant 1468/19 and NSF-BSF grant DMS-2247834. DL acknowledges the support of the Natural Sciences and Engineering Research Council of Canada and the Department of Mathematical and Statistical Sciences at the University of Alberta.
}

\maketitle
\section{Introduction and main results}
Owing to the striking resemblance between various phenomena in convex geometry and information theory, one is naturally led to the study of parallels between the two subjects. This direction of research goes back at least to the work of Costa and Cover \cite{CostaCover84} in the early 1980's, when they explicitly observed the similarity between the Brunn--Minkowski inequality in convex geometry and the Entropy Power inequality in information theory, both being cornerstone results in their respective fields. A connection between these inequalities was perhaps observed even earlier by Lieb, appearing implicitly in the work \cite{Lieb78}. From this point onwards, one finds a continuous stream of works elaborating the underlying dictionary between convex geometry and information theory while enriching both fields in the process (for example, \cite{DemboCoverThomas91, GuleryuzLutwakYangZhang02, BobkovMadiman12, FradeliziMarsiglietti14}). We shall refrain from giving a general overview of this vibrant field and refer the interested reader to the beautifully presented survey article \cite{MadimanMelbourneXu17} (and the references therein) for a bigger picture.

In the present work, we aim to supplement the evidence in favour of the Kneser--Poulsen conjecture, as well as offer a new potential tool, by observing that the mirroring entropy-version holds. Recall that, the Kneser--Poulsen conjecture in convex and discrete geometry asserts the intuitive statement that volume of the union of a finite collection of balls of a fixed radius cannot increase if their centres are brought closer together.
\begin{con}[Kneser--Poulsen, \cite{Kneser55, Poulson54}] \label{con: kporiginal}
Let $\{x_{1} , \ldots , x_{k}\}$ and $\{y_{1}, \cdots , y_{k}\}$ be two sets of points in $\mathbb{R}^n$ such that $\Vert y_{i} - y_{j} \Vert_{2} \leq \Vert x_{i} - x_{j} \Vert_{2}$ for all $i, j \in \{1, \ldots, k\}$. If $r> 0$, then we have:
\begin{equation} \label{eq: kporiginal}
\vol \left( \bigcup_{i=1}^{k} \ball (y_{i}, r) \right) \leq \vol \left( \bigcup_{i=1}^{k} \ball (x_{i}, r) \right),
\end{equation}
where $\vol$ is the Lebesgue measure on $(\R^n, \Vert \cdot \Vert_{2})$ and $\mathcal{B}(x,r)$ is the ball of radius $r$ centred at $x$.
\end{con}
\begin{rem}
    Sometimes a more general version is considered, where the radii of the balls are allowed to be different. However, we will restrict ourselves to the original formulation as stated above. 
\end{rem}
If $K = \{ x_{1} , \ldots , x_{k} \}$, the set $\bigcup_{i=1}^{k} \ball (x_{i}, r)$ can be rewritten as the Minkowski sum $K + r\mathcal{B}$, where $\ball$ denotes the unit ball $\ball (0, 1)$. By an elementary approximation-from-within argument, one can go from finite sets to compact sets in order to rephrase Conjecture \ref{con: kporiginal} in the manner below.
\begin{con} \label{con: kp}
For every contraction (that is, a $1$-Lipschitz map) $T: (K, \Vert \cdot \Vert_{2}) \to (\R^n, \Vert \cdot \Vert_{2})$ defined on a compact set $K 
\subseteq \R^n$, $r>0$, we have
\begin{equation}
\vol(T[K] + r \ball) \leq \vol (K + r \ball).
\end{equation}
\end{con}

Beyond the $n=2$ case which was resolved by Bezdek and Connelly \cite{BezdekConnelly02}, very little is known. For a general dimension $n$, the Kneser--Poulsen conjecture has been established in various cases, under rather strong restrictions on the set $K$ and the map $T$. For example, Csik\'os \cite{Csiskos98} proved Conjecture \ref{con: kporiginal} for continuous contractions (see Definition \ref{def: continuous contrac}). A modification of the formulas for volume obtained by Csik\'os play a key role in Bezdek and Connelly's proof in the plane, where they are delicately combined with an old trick (used previously, for example, by Alexander \cite{Alexander85}) to move the $x_{i}$ to the $y_{i}$ in a larger ambient space. Other such examples include the more recent work of Bezdek and Nasz\'odi \cite{BezdekNaszodi18} which demonstrates the conjecture when $k$ is large, for uniform contractions, that is, when there exists $\lambda > 0$ such that $\Vert y_{i} - y_{j} \Vert_{2} < \lambda < \Vert x_{i} - x_{j} \Vert_{2}$ for all $i \neq j$. In the same work, the authors also settle the case when the pairwise distances are reduced in every coordinate, for any $k$. For the current state-of-the-art regarding the Kneser--Poulsen conjecture, including many foundational works along with recent exciting developments that we have skipped here, we refer the reader to \cite{Bezdek13, KuperbergToth22}. However, despite the progress made so far, the Kneser--Poulsen conjecture remains largely out of reach. This is perhaps due to the fact that, at present, there is no way to deal with arbitrary contractions in a manner amicable to the geometric computations required in this context.  

In the pioneering comparison of the Brunn--Minkowski inequality and the Entropy Power inequality discussed earlier, one immediately observes that Euclidean balls correspond to Gaussian measures and (the logarithm of) volume corresponds to the Shannon--Boltzmann entropy. Indeed, on one hand, the Brunn--Minkowski inequality can be stated in the form 
\begin{equation} \label{eq: BMaddcombform}
\vol (A + B) \geq \vol (A^{\ast} + B^{\ast}),
\end{equation}
where $A^{\ast}, B^{\ast}$ denote Euclidean balls having the same volume as $A,B$, respectively. On the other hand, the Entropy Power inequality asserts that 
\begin{equation} \label{eq: EPIaddcombform}
h(X + Y) \geq h(X^{\ast} + Y^{\ast}), 
\end{equation}
where $h(\cdot)$ denotes the Shannon--Boltzmann entropy (see Definition \ref{def: entropy}), $X,Y$ are independent random vectors, $X^{\ast}, Y^{\ast}$ are independent standard Gaussian random vectors scaled to have the same entropy as $X,Y$, respectively. 

These correspondences, summarised in Table \ref{tab: dictionary}, indicate that the Shannon--Boltzmann entropy of the heat flow $X + \sqrt{s}Z$ is the natural information-theoretic analogue of the volume of a tube $ K + r \ball$. In this paper, we prove that the information-theoretic analogue mimics the geometric phenomena predicted by the Kneser--Poulsen conjecture \ref{con: kp}. Before we state our main results concretely, which are more general than the last claim, we set up some notation.

\begin{table}[ht]
    \centering
\begin{tblr}[caption={"sdsdssdfsdfds"}]{|[2pt,cyan3]c|[dashed,purple3]c|[2pt,cyan3]}
\hline[2pt,cyan3]
\textbf{Geometry}   & \textbf{Information theory}\\
\hline[purple3,dashed]
sets $K$ & random vectors $X$   \\
\hline[dotted]
Minkowski sum $K+L$  & independent sum $X+Y$     \\
\hline[dotted]
volume $\vol(K)$  & entropy $h(X)$     \\
\hline[dotted]
euclidean ball $\ball$   & standard Gaussian $Z$     \\
\hline[dotted]
 $\vol(K + r \ball)$ & $h(X + \sqrt{s}Z)$     \\
\hline[2pt,cyan3]
\end{tblr}
\caption{}
\label{tab: dictionary}
\end{table}

\subsubsection*{Some notation and definitions}
Throughout, unless stated otherwise, all sums $X + Y$ of random vectors will be sums of independent random vectors. The density of a random vector $X$ with respect to the Lebesgue measure (if it exists) will sometimes be denoted by $f_{X}$. The letter $Z$ is reserved for a random vector that has the standard Gaussian distribution in the ambient space, that is, if the discussion is in $\R^n$ we will have $f_{Z}(x) = \frac{1}{(2 \pi)^{n/2}}e^{- \frac{\Vert x \Vert^{2}_{2}}{2}}$. The word ``Gaussian'' will also be used for scalings of $Z$. We will write $\dom( f)$ for the domain of partially defined functions. For a set $K$, $\conv(K)$ will denote the convex hull of $K$. 

 \begin{definition} \label{def: entropy}
Let $X$ be an $\R^n$-valued random vector with density $f$ with respect to the Lebesgue measure. Then, the R\'enyi entropy of order $\alpha \in (0,1)\cup (1, \infty)$ of $X$ is given by,
\begin{equation}
h_{\alpha}(X) = \frac{1}{1 - \alpha} \log \int_{\R^n} f^{\alpha} \d x. 
\end{equation}
The R\'enyi entropy of orders $0$ and $\infty$ are obtained via taking respective limits,
\begin{equation}
\begin{split}
& h_{0}(X) = \lim_{\alpha \to 0^{+}} h_{\alpha}(X) = \log \vol (\supp(f)), \\ 
&h_{\infty}(X) = \lim_{\alpha \to \infty} h_{\alpha}(X)= - \log \Vert f \Vert_{\infty}. \\
\end{split}
\end{equation}
We define the R\'enyi entropy of order $1$ by,  
\begin{equation} \label{eq: defshannon}
h_{1}(X)=   -\int f \log f \d x, 
\end{equation}
if this integral exists. In this case, $h_{1}(X)$ is simply denoted by $h(X)$, and is called the Shannon--Boltzmann entropy.
\end{definition}
\begin{rem}
Conditions sufficient for $\lim_{\alpha \to 1^{+}} h_{\alpha} (X)$ or $\lim_{\alpha \to 1^{-}} h_{\alpha} (X)$ to equal $h_{1}(X)$ can be found in \cite[Lemma V.I]{MadimanWang14}. Moreover, when we write $h_{1}(X)$ (or $h(X)$) in this work (for example, in Theorems \ref{thm: main} and \ref{thm: entropickp}), it should be understood that we also assert the existence of the integral in Equation \eqref{eq: defshannon} (although it may take the value $+ \infty$). 
\end{rem}
For an $\R^{n}$-valued random vector $X$ with distribution $\mu$ having density $f$ with respect to the Lebesgue measure, we will abuse notation and use $h_{\alpha}(X), h_{\alpha} (\mu) ,$ and $h_{\alpha}(f)$, interchangeably. A property of R\'enyi entropies that we shall use later is that $h_{\alpha}(X,Y) = h_{\alpha}(X) + h_{\alpha}(Y)$ holds if $X,Y$ are independent vectors taking values in $\R^n, \R^m$, respectively. 
\begin{definition} \label{def: continuous contrac}
  Let $T : K \to \R^n$ be a contraction, i.e., a $1$-Lipschitz map $(K, \Vert \cdot \Vert_{2}) \to (\R^n, \Vert \cdot \Vert_{2})$. The map $T$ is said to be a continuous contraction if each point $x \in K$ can be joined to $T(x) \in T[K]$ by a curve $c_{x}: [0,1] \to \R^n$ such that $\Vert c_{x}(t) - c_{y}(t) \Vert_{2}$ is monotonically decreasing. In this case, we call $T_{t}: x \mapsto c_{x}(t)$ a continuously contracting family of maps, and the curve $c_{x}(t) = T_{t}(x)$ is called the trajectory of $x$. In particular, $T_{0}$ is the identity map.
\end{definition}
\subsection{The main result}
It is simple to check that $h_{\alpha}(X) \geq h_{\alpha}(T(X))$ if $T$ is a contraction. However, when the distributions of $X$ and $T(X)$ both undergo simultaneous heat flow to $X + \sqrt{s}Z$, and $T(X) + \sqrt{s}Z$, respectively, there may not be a contraction mapping the former to the latter. Our first main result says that their R\'enyi entropies may nevertheless be compared.
\begin{thm}\label{thm: main}
    Let $T: \R^n \to \R^n$ be a $1$-Lipschitz map. Then for every $\R^n$-valued random vector $X$, $s > 0$, the inequality  
    \begin{equation}
h_{\alpha}(X + \sqrt{s}Z) \geq h_{\alpha}(T(X) + \sqrt{s}Z),
    \end{equation}
    holds for all $\alpha \in [1, \infty]$, as well as for $\alpha \in [0,1)$ under the additional assumption that $\E \Vert X \Vert_{2} < \infty$. 
\end{thm}
\begin{rem}
    The above theorem is first proved under the condition $\E \Vert X \Vert_{2} < \infty$ for all $\alpha$. When $\alpha \geq 1$, we eliminate this condition by an approximation argument. For $\alpha \in [0,1)$, the approximation argument we use does not go through. Despite this, we are not aware of any counterexamples in the case $\alpha \in [0, 1)$ when $\E \Vert X \Vert_{2} = \infty$. 
\end{rem}
Note that Theorem \ref{thm: main} completely answers \cite[Question 1]{AishwaryaAlamLiMyroshnychenkoZatarainVera23} for Gaussian noise. Moreover, as explained in \cite{AishwaryaAlamLiMyroshnychenkoZatarainVera23}, a result such as Theorem \ref{thm: main} for the uniform distribution on a ball rather than the Gaussian would immediately yield the Kneser--Poulsen conjecture. 

The above theorem is obtained as a consequence of a stronger result that holds for continuous contractions under mild conditions on the trajectories.
\begin{thm} \label{thm: maincont}
    Suppose $T_{t}$ is a continuously contracting family of maps starting in a set $K \subset \R^n$ (that is, $T_{0}$ is defined on $K$) with smooth trajectories, and $X$ is a $K$-valued random vector with $\E \Vert X \Vert_{2} < \infty$. Then, if $\E \sup_{t} \Vert \frac{\d}{\d t}T_{t}(X) \Vert_{2} < \infty$, the following holds: for any convex function $\phi: [0, \infty) \to \R$ satisfying $\phi (0) = 0$, 
    \begin{equation}
    \int \phi (f_{T_{t}(X) + \sqrt{s}Z}) \d x
    \end{equation}
    is monotonically increasing in $t$.
\end{thm}
\begin{rem}
    Continuous contractions induce totally-ordered curves with respect to a majorisation order \cite[Lemma 1.5]{MelbourneRoberto23}. The second part of our result says that such curves continue to be totally-ordered under the action of Gaussian convolution. 
\end{rem}

Continuous contractions appear quite naturally in probability theory. The (reverse) Ornstein--Uhlenbeck flow on strongly log-concave measures \cite{KimMilman12}, the (reverse) heat flow for log-concave measures \cite{KlartagPutterman23}, and the maps that induce displacement interpolation (for optimal transport with quadratic cost) when the final transport map is $1$-Lipschitz--- are all examples of continuous contractions. The integrability condition $\E \sup_{t} \Vert \frac{\d}{\d t}T_{t}(X) \Vert_{2} < \infty$ is also fairly mild. For example, in the setup of the optimal transport problem with quadratic cost, it is always satisfied.

Of course, the information-theoretic analogue of the Kneser--Poulsen conjecture is the special case $\alpha=1$ of Theorem \ref{thm: main}.
\begin{thm}[The entropic Kneser--Poulsen theorem] \label{thm: entropickp}
    For every $1$-Lipschitz map $T: \R^n \to \R^n$ and random vector $X$, we have $h(X + \sqrt{s}Z) \geq h(T(X) + \sqrt{s}Z)$.
\end{thm}
Perhaps from an adjacent viewpoint, the Kneser--Poulsen theorem for information transmission would be the following statement.
\begin{cor} \label{cor: kpacity}
    Suppose Alice wants to communicate with Bob using the alphabet $K = \{ x_{1} , \cdots , x_{k} \}$ over a noisy channel with additive white Gaussian noise: Bob receives the random point $x + Z$ when Alice sends $x$. The optimal rate at which information can be reliably transmitted across this noisy channel cannot be improved by bringing points in $K$ pairwise closer.
\end{cor}
In information theory, this optimal rate is called the channel capacity $\mathcal{C}$ (see \cite{Kemperman69} for a quick mathematical introduction). Shannon, in his landmark work \cite{Shannon48}, showed that this quantity has a clean mathematical expression which translates to the following in the setting of Corollary \ref{cor: kpacity}:
\begin{equation} \label{eq: capacityformula}
\mathcal{C} = \sup_{X} I (X; X + Z),
\end{equation}
where the supremum is over all random vectors taking values in $K$, and the mutual information $I(X; X + Z) = h(X+Z) - h(X+Z \vert X) = h(X+Z) - h(Z)$ measures the amount of information shared between the ``input''$X$ and the ``output'' $X + Z$. Theorem \ref{thm: entropickp} shows $I (X; X + Z) \geq I(T(X); T(X)+Z)$ thereby proving Corollary \ref{cor: kpacity} in a pointwise-sense. We suspect that Corollary \ref{cor: kpacity} may hold in more generality, for every radially-symmetric log-concave noise $W$ instead of the Gaussian $Z$. It is possible to approach the Kneser--Poulsen conjecture directly based on R\'enyi-generalisations of channel capacity, but this will be pursued elsewhere. 

The Entropy Power inequality \eqref{eq: EPIaddcombform} can be strengthened if one of the summands is Gaussian. This is a deep result of information theory discovered by Costa \cite{Costa85}. We state it in terms of the entropy power $N(\cdot) = e^{\frac{2 h(\cdot)}{n}}$ below. 
\begin{thm} \cite{Costa85} \label{thm: costaEPI} 
    For any $\R^n$-valued random vector $X$ such that $N(X)$ exists, $N(X + \sqrt{s}Z)$ is a concave function of $s$. 
\end{thm}
Simple algebraic manipulations (see, for example, \cite{Costa85} or \cite[Section 2.2.2]{MadimanMelbourneXu17}) show that the concavity of $N(X + \sqrt{s}Z)$ in $s$ is equivalent to the assertion that $A(\beta) = N(\beta X + Z ) - N (\beta X)$ is increasing in $\beta \in [0,1]$. The inequality $A(1) \geq A(0)$ says $N(X+Z) \geq N(X) + N(Z)$, which is exactly the Entropy Power inequality when $Y=Z$ in \eqref{eq: EPIaddcombform} since $N(\cdot)$ is additive over the class of scaled standard Gaussians. In this setting, Theorem \ref{thm: entropickp} and Theorem \ref{thm: costaEPI} admit the following unification.  
\begin{thm} \label{thm: entropickp+costa}
For every Lipschitz map $T: \R^n \to \R^n$ with Lipschitz constant $\lip(T) \leq 1$, and every random vector $X$ for which $N(X)$ exists, 
\begin{equation} \label{eq: entropickp+costa}
    N(X + Z) \geq N(T(X) + Z) + (1- \lip^{2} (T)) N(X).
\end{equation}
\end{thm}
On one hand, applying Theorem \ref{thm: entropickp+costa} to maps $T: x \mapsto \beta x$, $\beta \in [0,1]$, yields Theorem \ref{thm: costaEPI}. On the other hand, if $\lip (T) \leq 1$ we clearly obtain the conclusion of Theorem \ref{thm: entropickp}. Theorem \ref{thm: entropickp+costa} was anticipated in \cite[Question 2]{AishwaryaAlamLiMyroshnychenkoZatarainVera23}, where the special case of linear $T$ was proven. 
\begin{rem}
    The geometric counterpart of Costa's result, namely the concavity of $\vol (K + r \ball)^{\frac{1}{n}}$ in $r$, does not hold for every compact $K \subseteq \R^n$ if $n > 1$ \cite{FradeliziMarsiglietti14}. However, as shown in \cite{CostaCover84}, it follows from the Brunn--Minkowski inequality \eqref{eq: BMaddcombform} that $\vol (K + r \ball)^{\frac{1}{n}}$ is indeed concave in $r$ if $K$ is convex. Therefore, using a similar argument as employed by us to prove Theorem \ref{thm: entropickp+costa}, the following inequality is implied by the Kneser--Poulsen conjecture if $K$ is assumed to be convex:
    \[
\vol(K + \ball)^{\frac{1}{n}} \geq \vol(T[K] + \ball)^{\frac{1}{n}} + (1 - \lip (T) )\vol (K)^{\frac{1}{n}}.
    \]
\end{rem}

\subsection{Plan of proof of the main results} We first prove Theorem \ref{thm: maincont} using a mass-transport argument. Given a curve $\{\mu_{t}\}_{t \in [0,1]}$ of probability measures and a velocity-field $v_{t}$ compatible with it, in Proposition \ref{prop: velocityforconvolution} we obtain a formula for a velocity-field $\tilde{v}_{t}$ compatible with the noise-perturbed curve $\{ \mu_{t} \star \nu \}_{t \in [0,1]}$, where $\nu$ is any measure with a bounded Lipschitz smooth density. In Section \ref{sec: proofmain} we show, under the assumption $\nu = \gamma$ is the standard Gaussian measure, that if $T_{t}$ is continuously contracting, $\mu_{0}$ any probability measure, $\mu_{t} = {T_{t}}_{\#}\mu_{0}$, then $\grad \cdot \tilde{v}_{t} \leq 0$. Proposition \ref{prop: divandent} allows us to deduce Theorem \ref{thm: maincont} from this divergence condition. Theorem \ref{thm: main} elegantly follows from Theorem \ref{thm: maincont}, using the same old trick of extending the phase space. Only this time, thanks to the tensorisation properties of both the Gaussian measure and R\'enyi entropies, the argument goes through in all dimensions effortlessly. Of course, Theorem \ref{thm: entropickp} is contained in Theorem \ref{thm: main}. A short proof of Theorem \ref{thm: entropickp+costa} is included in Section \ref{sec: proofmain}.

\subsection{Acknowledgements}
We express our heartfelt gratitude to Irfan Alam, Serhii Myroshnychenko, and Oscar Zatarain-Vera, for many enriching discussions around the Kneser--Poulsen theme. We are indebted to Boaz Slomka for introducing us to the Kneser--Poulsen conjecture. The authors would like to thank two anonymous referees for their feedback, which helped us improve the presentation as well as technical aspects of this paper.
 
\section{Preparation: curves of probability measures, perturbations by convolutions} \label{sec: prep}
Let $\{ \mu_{t} \}_{t \in [0,1]}$ be a curve of probability measures in $\R^n$. We say that a time-dependent velocity-field $\{ v_{t} \}_{t \in [0,1]}$ is compatible with $\{ \mu_{t} \}_{t \in [0,1]}$ if the continuity equation
\begin{equation} \label{eq: conteq}
\partial_{t} \mu_{t} + \grad \cdot (v_{t} \mu_{t}) = 0
\end{equation}
is satisfied in the weak sense. The latter equation means that 
\begin{equation} \label{eq: conteqweak}
 \frac{\d }{\d t} \int f \d \mu_{t} = \int \langle \grad f , v_{t} \rangle \d \mu_{t},
\end{equation}
for all test functions $f$. Throughout the paper, by a test function we mean a compactly supported smooth function defined on $\R^n$. Once Equation \eqref{eq: conteq} is known to hold, Equation \eqref{eq: conteqweak} holds under wider generality; for example, it holds for all bounded Lipschitz smooth functions (see \cite[Chapter 8]{AmbrosioGigliSavare08}). 

\begin{exmpl} \label{ex: main}
Curves of probability measures and compatible velocity-fields arise naturally in the following manner. Suppose a one-parameter family of maps $T_{t}$ is given, with the property that the trajectory $T_{t}(x)$ of each point $x$ is smooth and $v_{t}$ is well-defined by the equation $\frac{\d}{\d t} T_{t}(x) = v_{t} (T_{t}(x))$. Fix a probability measure $\mu_{0}$, define $\mu_{t} = {T_{t}}_{\#} \mu_{0}$. Assume that the $\sup_{t} \Vert v_{t} \circ T_{t} \Vert_{2} \in L^{1}(\mu_{0})$. This allows for the computation below to hold for any test function $f$: 
\begin{equation}
\begin{split}
&\frac{\d}{\d t} \int f (x) \d \mu_{t} (x) = \frac{\d}{\d t} \int f (T_{t}(x)) \d \mu_{0}(x) = \int \frac{\d}{\d t} f(T_{t}(x)) \d \mu_{0}(x) \\
&= \int \Big\langle \grad f(T_{t}(x)) , \frac{\d}{\d t}T_{t}(x) \Big\rangle \d \mu_{0}(x) =  \int \Big\langle \grad f(T_{t}(x)) , v_{t}(T_{t}(x) )\Big\rangle \d \mu_{0}(x) \\
&= \int \langle \grad f (x), v_{t} (x) \rangle \d \mu_{t} (x),
\end{split}
\end{equation}
where the assumption on $\Vert v_{t} \circ T_{t} \Vert_{2}$ is used to justify differentiation under the integral sign.
Thus, the curve of probability measures $\mu_{t}$ and the time-dependent velocity-field $v_{t}$ in this setup are compatible.  
\end{exmpl}
\begin{prop} \label{prop: velocityfieldforcontcontrac}
    A continuously contracting family of maps $\{T_{t} \}_{t \in [0,1]}$ with smooth trajectories has a well-defined velocity-field $v_{t}$. Moreover, $v_{t}$ has the monotonicity property $\langle v_{t}(x) - v_{t} (y) , x - y \rangle \leq 0$ on the support of $v_{t}$.
\end{prop}
\begin{proof}
The existence of a well-defined velocity field follows simply from the fact that if two trajectories $T_{t}(x)$ and $T_{t}(y)$ meet at time $t_{0} \in (0,1)$, then one must necessarily have $T_{t}(x) = T_{t}(y)$ for all $t \geq t_{0}$ because the $T_{t}$ are continuously contracting. 

For the monotonicity property, let $x_{0}$ and $y_{0}$ be such that $T_{t}(x_{0}) = x$ and $T_{t}(y_{0}) = y$, then 
\begin{equation}
\begin{split}
0 \geq \frac{\d }{\d t} \Vert T_{t}(x_{0}) - T_{t}(y_{0}) \Vert_{2}^{2} &= 2 \langle v_{t}(T_{t}(x_{0})) - v_{t}(T_{t}(x_{0})) , T_{t}(x_{0}) - T_{t}(y_{0}) \rangle \\
&= 2 \langle v_{t}(x) - v_{t} (y) , x - y \rangle . \\
\end{split}
\end{equation}
\end{proof} 
In the terminology of \cite{Minty61}, a set $E \subseteq \R^n \times \R^n$ is called totally-M-related if $\langle x_{1} - x_{2} , y_{1} - y_{2} \rangle \geq 0$ for all $(x_{1}, y_{1}) , (x_{2}, y_{2}) \in E$. Let us call the set $\{ x \in \R^n \, : \, \textnormal{there exists }y \textnormal{ such that }(x,y) \in E \}$ the domain of $E$, denoted $\dom (E)$. The main result of \cite{Minty61} says: 
\begin{thm} \cite{Minty61} \label{thm: mintylater}
    Let $E \subseteq \R^n \times \R^n$ be a maximal totally-M-related set. Then $\dom (E)$ contains the relative interior of $\conv ( \dom (E) )$. 
\end{thm}
Recall that relative interior of a set means the interior of the set taken with respect to its affine span. Minty's proof of the above theorem goes via showing that, for any totally-M-related set $E$ and any point $x_{0}$ in the relative interior of $\conv ( \dom (E) )$, there exists a $y_{0}$ such that $E \cup \{ (x_{0}, y_{0}) \}$ is also totally-M-related. This is shown by combining another result of Minty, stated below, with a compactness argument. 
\begin{thm} \cite[Theorem 1]{Minty62}
    Let $x_{1} , \ldots , x_{m}$ and $y_{1} , \ldots , y_{m}$ be points in $\R^n$ such that $\langle x_{i} - x_{j} , y_{i} - y_{j} \rangle \geq 0$ for $i=1, \ldots, m$. Then, for every $x \in \R^n$ there exists a $y \in \R^n$ for which $\langle x - x_{i} , y - y_{i} \rangle \geq 0$. 
\end{thm}
In our setting, for any fixed $t$, the set $E = \{ (x, - v_{t}(x)) \, : \,  x \in \support (v_{t}) \}$ is totally-M-related, by Proposition \ref{prop: velocityfieldforcontcontrac}. Thus, Minty's proof in \cite{Minty61} allows us to extend $v_{t}$ to any point in the relative interior of $\conv ( \dom (v_{t}) )$, while preserving the monotonicity property in Propostion \ref{prop: velocityfieldforcontcontrac}. This observation is crucially used in the proof of Theorem \ref{thm: maincont}.

Given a curve of probability measures, the dynamics of its noise-perturbation is described below.
\begin{prop} \label{prop: velocityforconvolution}
    Let $\{\mu_{t}\}$ be a curve of probability measures and $v_{t}$ a time-dependent velocity-field compatible with it. Suppose $\nu$ is a measure having a bounded Lipschitz smooth density. Then, a time-dependent velocity field compatible with the curve $\tilde{\mu_{t}} = \mu_{t} \star \nu$ is given by the conditional expectation
    \begin{equation}
\tilde{v_{t}}(x) = \E \left[ v_{t}(X_{t}) \vert X_{t} + Y = x \right],
    \end{equation}
    where $X_{t} \sim \mu_{t}$, $Y \sim \nu$ is independent of the $X_{t}$. 
\end{prop}
\begin{proof}
  Let $f$ be a test function. Suppose $\nu$ has density $g$ with respect to the Lebesgue measure. Then,
  \begin{equation}
\begin{split}
    \frac{\d }{\d t} \int f(x) \d (\mu_{t} \star \nu)(x) 
    &= \frac{\d}{\d t } \int \left[ \int f(x) g(x-y)  \d x  \right] \d \mu_{t}(y) \\
    &= \int \bigg\langle \grad_{y} \left( \int f(x) g (x-y) \d x \right) , v_{t}(y) \bigg\rangle  \d \mu_{t}(y) \\
    &= \int \bigg\langle -\int f(x) (\grad g)(x-y) \d x , v_{t} (y) \bigg\rangle  \d \mu_{t}(y) \\
    &= \int \bigg\langle \int \grad f (x) g(x-y) \d x , v_{t} (y) \bigg\rangle  \d \mu_{t}(y) \\
    &= \int \bigg\langle \grad f (x)  , \int v_{t} (y) g(x-y) \d \mu_{t}(y) \bigg\rangle \d x \\
    &= \int \bigg\langle \grad f (x) , \frac{\int v_{t}(y) g (x-y) \d \mu_{t}(y)}{\mu_{t} \star g (x)} \bigg\rangle  \mu_{t} \star g (x) \d x \\
    &= \int \langle \grad f (x) ,  \E \left[ v_{t}(X_{t}) \vert X_{t} + Y = x \right] \rangle \d (\mu_{t} \star \nu)(x).
\end{split}
  \end{equation}
  Thus, the continuity equation $\partial_{t} \tilde{\mu_{t}} + \grad \cdot (\tilde{v_{t}} \tilde{\mu_{t}}) = 0$ is verified. Here the conditions on $g$ are used to justify the second equality.
\end{proof}
We need one more ingredient to allow us to conclude entropic inequalities from compatible velocity-fields.
\begin{prop} \label{prop: divandent}
    Let $\mu_{t}$ be a curve in the space of probability measures and a time-dependent velocity field $v_{t}$ compatible with it. Assume that
    \begin{enumerate}
    \item $\int_{0}^{1} \int \Vert v_{t} \Vert_{2} \d \mu_{t} \d t < \infty$, and
    \item $v_{t}$ is locally-Lipschitz. 
    \end{enumerate}
    Suppose each $\mu_{t}$ has density $\rho_{t}$ with respect to the Lebesgue measure. Then, if $\grad \cdot v_{t} \leq 0$ for all $t$, then $ \int \phi (\rho_{t}) \d x$ is monotonically increasing in $t$ for every convex function $\phi$ defined on $[0, \infty)$ such that $\phi(0) = 0 $.
\end{prop}
\begin{proof}
    The regularity assumptions on $v_{t}$ imply the existence of $T_{t}$ such that $\mu_{t} = {T_{t}}_{\#} \mu_{0}$ (see, for example, \cite[Page 15]{Villani09} or \cite[Proposition 8.1.8]{AmbrosioGigliSavare08}). Further, we have $\grad \cdot v_{t} \leq 0 $ which means that the $T_{t}$ are volume-contracting (that is, $\vert \det \grad T_{t} \vert \leq 1$). 

    The proof proceeds by majorisation as defined in \cite[Definition 5]{AishwaryaAlamLiMyroshnychenkoZatarainVera23}. In light of \cite[Lemma 2.2]{AishwaryaAlamLiMyroshnychenkoZatarainVera23} and \cite[Lemma 2.1]{AishwaryaAlamLiMyroshnychenkoZatarainVera23}, to show $\int \phi (\rho_{0}) \d x \leq \int \phi (\rho_{t}) \d x $ it suffices to show that for every compact $K_{0}$ there exists a compact set $K_{t}$ with volume at most $\vol(K_{0})$, and satisfying $\mu_{0}(K_{0}) \leq \mu_{t}(K_{t})$. In our case, such a $K_{t}$ is obtained by setting $K_{t} = T_{t}[K_{0}]$. This is because $X_{0} \in K_{0}$ implies $T_{t}(X_{0}) \in T_{t}[K_{0}]$, where $X_{0}$ is a random vector with distribution $\mu_{0}$. A similar argument can be employed to show $\int \phi (\rho_{t_{0}}) \d x \leq \int \phi (\rho_{t_{1}}) \d x$ for any $t_{0} \leq t_{1}$.
\end{proof}  
\section{Proof of the main theorems} \label{sec: proofmain}
We will first prove Theorem \ref{thm: maincont}.
\begin{proof}[Proof of Theorem \ref{thm: maincont}]
Let $T_{t}$ be a family of continuously contracting maps in $\R^n$ having smooth trajectories. Call the corresponding velocity field $v_{t}$. Let $X$ a random vector such that $\E \sup_{t} \Vert \frac{\d}{\d t} T_{t}(X) \Vert_{2} < \infty$. Denote by $\mu_{0}$ the distribution of $X$, and set $\mu_{t} := {T_{t}}_{\#} \mu_{0}$ so that $\mu_{t}$ is the distribution of $T_{t}(X)$. The computation in Example \ref{ex: main} shows that the velocity-field $v_{t}$ is compatible with $\mu_{t}$. Consider $\tilde{\mu_{t}} = \mu_{t} \star \nu$, where $\d \nu = g \d x = e^{-V} \d x$. We are interested in $\nu$ being the distribution of $\sqrt{s}Z$, and thus $g = C_{s}e^{-\frac{\Vert x \Vert_{2}^{2}}{2s}}, \grad g(x) = - g (x) \grad V (x) =  - g (x) \frac{x}{s}$. We shall write the proof for $s=1$, the case of a general $s$ works in exactly the same manner. 

The density $g$ is easily seen to be bounded, Lipschitz, and smooth. Thus, by Proposition \ref{prop: velocityforconvolution}, we obtain a velocity-field compatible with $\tilde{\mu_{t}} $: 
\begin{equation}
\begin{split}
    \tilde{v_{t}}(x) = \frac{\int v_{t}(y) g (x - y) \d \mu_{t}(y)}{\int  g (x - y) \d \mu_{t}(y)}.  
\end{split}
\end{equation}
Hence,
\begin{equation}
\begin{split}
\grad \tilde{v_{t}}(x) = &\frac{\int v_{t}(y) 
 \otimes \grad_{x} \left( g (x - y) \right) \d \mu_{t}(y) }{\int  g (x - y) \d \mu_{t}(y)} \\
&- \frac{\int \grad_{x} \left( g (x - y) \right) \d \mu_{t}(y) \otimes \int v_{t}(y) g (x - y) \d \mu_{t}(y)}{\left( \int  g (x - y) \d \mu_{t}(y) \right)^{2}}\\
= & - \frac{\int v_{t}(y) 
 \otimes \grad V (x - y) g (x - y) \d \mu_{t}(y) }{\int  g (x - y) \d \mu_{t}(y)} \\
& +  \frac{\int  \grad V (x - y)g (x - y) \d \mu_{t}(y) \otimes \int v_{t}(y) g (x - y) \d \mu_{t}(y)}{\left( \int  g (x - y) \d \mu_{t}(y) \right)^{2}}.\\
\end{split}
\end{equation}
Let $Y = Y_{t,x}$ be a random vector with density $ f_{x}(y) = \frac{g(x-y)}{ \int g(x-y) \d\mu_{t}(y)}$ with respect to $\mu_{t}$. Then, by taking trace,
\begin{equation} \label{eq: divergencecomputation}
\begin{split}
\grad \cdot \tilde{v}_{t} (x) 
&= - \E \langle v_{t}(Y) , \grad V (x-Y) \rangle  + \langle \E \grad V (x-Y) , \E v_{t}(Y) \rangle \\
&= - \E \langle v_{t} (Y) , x \rangle + \E \langle v_{t}(Y), Y \rangle + \langle x , \E v_{t} (Y) \rangle - \langle \E Y , \E v_{t}(Y) \rangle \\
&= \E \langle v_{t} (Y) , Y \rangle  - \langle \E Y , \E v_{t}(Y) \rangle \\
&= \E \langle v_{t}(Y) - \E v_{t}(Y) , Y- \E Y \rangle.
\end{split}
\end{equation}
Now, since $\E (Y - \E Y) = 0$, $\E v_{t} (Y)$ can be replaced with any constant of choice. We extend $v_{t}$ to $\E Y$ while preserving its monotonicity property established in Proposition \ref{prop: velocityfieldforcontcontrac}. This can be done using Minty's results discussed in Section \ref{sec: prep}. Then, we choose the constant replacing $\E v_{t} (Y)$ to be $v_{t}(\E Y)$, so that 
\begin{equation}
\grad \cdot \tilde{v_{t}} (x) = \E \langle v_{t} (Y) - v_{t} (\E Y) , Y - \E Y \rangle \leq 0.
\end{equation}
Further, 
\[
\begin{split}
\int \Vert \tilde{v}_{t} \Vert_{2} \d \tilde{\mu}_{t}  &= \int \left\Vert  \frac{\int v_{t}(y) g (x - y) \d \mu_{t}(y)}{\int  g (x - y) \d \mu_{t}(y)} \right\Vert_{2} \d \tilde{\mu}_{t} (x)  \\
&\leq \int \Vert v_{t}(y) \Vert_{2} \d \mu_{t} (y) ,
\end{split}
\]
so that $\int_{0}^{1} \int \Vert \tilde{v}_{t} \Vert_{2} \d \tilde{\mu}_{t} \d t < \infty$ follows from $\E \sup_{t} \Vert \frac{\d}{\d t} T_{t}(X) \Vert_{2} < \infty$. The explicit formula for $\grad \tilde{v}_{t}$ makes it clear that $\tilde{v}_{t}$ is locally-Lipschitz in the space variable. Hence, $\tilde{\mu}_{t}$ and $\tilde{v}_{t}$ satisfy the hypothesis required in Proposition \ref{prop: divandent}, thereby establishing Theorem \ref{thm: maincont}.
\end{proof}
Now we use Theorem \ref{thm: maincont} to prove Theorem \ref{thm: main}.

\begin{proof}[Proof of Theorem \ref{thm: main}]
Suppose an arbitrary $1$-Lipschitz map $T: \R^n \to \R^n$ is given. We begin by showing that it suffices to prove Theorem \ref{thm: main} for $s=1$. Note that, $\tilde{T}(x) = \frac{1}{\sqrt{s}} T (\sqrt{s}x)$ is also $1$-Lipschitz. By using the scaling property $h_{\alpha}(\lambda X) = h_{\alpha}(X) + n \log |\lambda|$ and Theorem \ref{thm: main} for $s=1$, we have Theorem \ref{thm: main} for arbitrary $s > 0$:
\begin{equation}
\begin{split}
    h_{\alpha}(X + \sqrt{s}Z) &= h_{\alpha}\left(\sqrt{s} \left( \frac{1}{\sqrt{s}}X + Z \right) \right) = h_{\alpha} \left( \frac{1}{\sqrt{s}}X + Z \right) + \frac{n}{2} \log s \\
    &\geq h_{\alpha} \left( \tilde{T} \left( \frac{1}{\sqrt{s}}X \right) + Z \right) + \frac{n}{2}\log s \\
    &= h_{\alpha} \left( \frac{1}{\sqrt{s}}T(X) + Z \right) + \frac{n}{2} \log s = h_{\alpha}(T(X) + \sqrt{s}Z).
\end{split}    
\end{equation}
Thus, without loss of generality, we assume that $s=1$ for the remainder of this proof.

Suppose $\E \Vert X \Vert_{2} < \infty$. We use the trajectories from \cite[Lemma 1]{BezdekConnelly02},
\begin{equation}
    S_{t}(x,0) = \left( \frac{x + T(x)}{2} + (\cos \pi t) \frac{x - T(x)}{2} , (\sin \pi t) \frac{x - T(x)}{2} \right).
\end{equation}
This defines a continuously contracting family of maps which takes $(x,0) \in \R^{2n}$ to $(T(x),0) \in \R^{2n}$ from time $t=0$ to $t=1$. Since $\E \Vert X \Vert_{2} < \infty$, it is easily verified that $\E \sup_{t} \Vert \frac{\d}{\d t}S_{t}(X,0) \Vert_{2} < \infty$. Therefore, Theorem \ref{thm: maincont} can be invoked to obtain
\begin{equation}
h_{\alpha} ((X,0) + (Z, Z')) \geq h_{\alpha} ((T(X),0) + (Z, Z')) ,
\end{equation}
where $Z, Z'$ are independent standard Gaussians in $\R^n$ so that $(Z, Z')$ is the standard Gaussian random vector in $\R^{2n}$. Using the tensorisation property of R\'enyi entropies,
\begin{equation}
    h_{\alpha}((X,0) + (Z, Z') ) = h_{\alpha}((X + Z,Z')) = h_{\alpha} (X + Z) + h_{\alpha}(Z'), 
\end{equation}
and similarly $h_{\alpha}((T(X),0) + (Z, Z') ) =  h_{\alpha} (T(X) + Z) + h_{\alpha}(Z')$. 
Thus, we have
\[
h_{\alpha} (X + Z) \geq h_{\alpha} (T(X) + Z),
\]
under the assumption $\E \Vert X \Vert_{2} < \infty$. 

To drop the assumption $\E \Vert X \Vert_{2} < \infty$ for $\alpha \in [1, \infty]$ we employ the following truncation argument. Suppose $X$ is an arbitrary $\R^{n}$-valued random vector. Define $X_{k} = X \ind_{\Vert X \Vert_{2} \leq k}$. Then, $X_{k}$ converges weakly to $X$, and $T(X_{k})$ converges weakly to $T(X)$. Thus, $f_{X_{k} + Z}(x) = \E g (x - X_{k})$ and $f_{T(X_{k}) + Z}(x) = \E g (x - T(X_{k}))$ converge to $f_{X + Z}(x)$ and $f_{T(X) + Z}(x)$, respectively. Moreover, $f_{X_{k} + Z}(x), f_{T(X_{k}) + Z}(x) \leq 1$ because $g \leq 1$. Therefore,
for $\alpha >1$, $f^{\alpha}_{X_{k} + Z}(x) \leq f_{X_{k} + Z}(x) $ and $f^{\alpha}_{T(X_{k}) + Z}(x) 
 \leq f_{T(X_{k}) + Z}(x) $. By a variant of the dominated convergence theorem \cite[Section 4.4]{Royden10}, we deduce that 
$\int f^{\alpha}_{X_{k} + Z}(x) \d x  \rightarrow \int f^{\alpha}_{X + Z}(x) \d x$ and $\int f^{\alpha}_{T(X_{k}) + Z}(x) \d x  \rightarrow \int f^{\alpha}_{T(X) + Z}(x) \d x$. That is, $h_{\alpha}(X_{k} + Z) \to h_{\alpha}(X + Z)$ and $h_{\alpha}(T(X_{k}) + Z) \to h_{\alpha}(T(X) + Z)$ as $k \to \infty$, if $\alpha > 1$. By construction, $\E \Vert X_{k} \Vert_{2} \leq k < \infty$ so we already have $h_{\alpha} (X_{k} + Z) \geq h_{\alpha} (T(X_{k}) + Z)$. The desired inequality for $X$ when $\alpha >1$ follows by taking limits. The $\alpha=1$ case follows from \cite[Lemma 5.1 (ii)]{MadimanWang14} since $h_{\alpha} (X + Z) \geq h_{\alpha} (Z) > - \infty$ and $h_{\alpha} (T(X) + Z) \geq h_{\alpha} (Z) > - \infty$ for all $\alpha > 1$. 

\end{proof}
\begin{rem}
    When $n=1$, since ``divergence $=$ derivative'', much stronger conclusions can be drawn by looking at the continuous contractions-part of the proof:
    \begin{enumerate}
    \item The proof works for arbitrary log-concave probability densities $g$ (that is, whenever $V$ is convex). This requires an application of Chebyshev's other inequality \cite{FinkJodeit84} after the first equality in Equation \eqref{eq: divergencecomputation}.
     \item The transport maps induced by the velocity fields $\tilde{v}_{t}$ are contractions.     
    \end{enumerate}
    These one-dimensional results are already known albeit with a different proof \cite{LewisThompson81}.
\end{rem}

\begin{proof}[Proof of Theorem \ref{thm: entropickp+costa}]
    Let $T$ be a Lipschitz map. Denote $\lip(T) = L \leq 1$. Note that Theorem \ref{thm: entropickp} implies 
    \begin{equation} \label{eq: entropickpL}
N (LX + Z) \geq N(T(X) + Z).
    \end{equation}
    To see this, apply Theorem \ref{thm: entropickp} to the $1$-Lipschitz map $T/L: x \mapsto \frac{1}{L}T(x)$ and $s= \frac{1}{L^{2}}$:
    \[
\frac{1}{L^{2}} N(LX + Z) = N \left( X + \frac{1}{L}Z \right) \geq N \left( \frac{1}{L}T(X) + \frac{1}{L}Z \right) = \frac{1}{L^{2}} N (T(X) + Z).
    \]
Theorem \ref{thm: costaEPI} gives 
\begin{equation} \label{eq: costaincreasing}
N(X+Z)-N(X) \geq N(LX + Z) - N(LX) = N(LX + Z) - L^{2} N(X).
\end{equation}
Combining Equations \eqref{eq: entropickpL} and \eqref{eq: costaincreasing}, we get
\[
N(X + Z) - N(X) \geq N(T(X) + Z ) - L^{2} N(X),
\]
as desired. 
   \end{proof}

\bibliographystyle{amsplain}
\bibliography{shoulderofgiants}
\end{document}